\documentclass[preprint,12pt]{elsarticle}

\usepackage{hyperref,amsmath,amsfonts,mathrsfs}

\usepackage{graphicx}
\usepackage{amsfonts}
\usepackage{amsbsy}
\usepackage{amssymb}
\usepackage{amsmath,tikz,pgfplots}
\usetikzlibrary{decorations.pathmorphing}
\usepackage{tikz}
\numberwithin{equation}{section}
\usetikzlibrary{arrows.meta, positioning}
\usepackage{graphics}
\usepackage{graphicx}
\usepackage{amsfonts}
\usepackage{amsbsy}
\usepackage{amssymb}
\usepackage{amsmath,tikz,pgfplots}
\usetikzlibrary{decorations.pathmorphing}
\usepackage{graphics}
\def\bpsp{\begin{pspicture}}
\def\epsp{\end{pspicture}}

\newtheorem{theorem}{Theorem}[section]
\newtheorem{remark}[theorem]{Remark}
\newtheorem{example}[theorem]{Example}
\newtheorem{lemma}[theorem]{Lemma}
\newtheorem{corollary}[theorem]{Corollary}
\newtheorem{definition}[theorem]{Definition}
\newtheorem{proposition}[theorem]{Proposition}
\newtheorem{note}{Note}
\newtheorem{case}{Case}
\newtheorem{conjecture}{Conjecture}
\newtheorem{question}{Question}

\newcommand{\bea}{\begin{eqnarray}}
\newcommand{\eea}{\end{eqnarray}}
\newcommand{\beq}{\begin{eqnarray*}}
\newcommand{\eeq}{\end{eqnarray*}}

\def\m4{\mbox{\rm ~(mod $4$)}}

\def \bd{\begin{definition}}
\def \ed{\end{definition}}
\def \bqu{\begin{question}}
\def \equ{\end{question}}
\def \bcc{\begin{conjecture}}
\def \ecc{\end{conjecture}}
\def \bt{\begin{theorem}}
\def \et{\end{theorem}}
\def \bl{\begin{lemma}}
\def \el{\end{lemma}}
\def \bc{\begin{corollary}}
\def \ec{\end{corollary}}
\def \be{\begin{equation}}
\def \ee{\end{equation}}
\def \ben{\begin{enumerate}}
\def \een{\end{enumerate}}
\def \ba{\begin{array}}
\def \ea{\end{array}}
\def \bp{\begin{proposition}}
\def \ep{\end{proposition}}
\def \bx{\begin{example}}
\def \ex{\end{example}}
\def \br{\begin{remark}}
\def \er{\end{remark}}
\def \bdsc{\begin{description}}
\def \edsc{\end{description}}

\def \bn{\begin{case}}
\def \en{\end{case}}
\def \bnt{\begin{note}}
\def \ent{\end{note}}
\def\1{1\!\!1}

\def\mm2{\mbox{\rm ~(mod $2$)}}
\def\m4{\mbox{\rm ~(mod $4$)}}

\def\qed{\nolinebreak\hfill\rule{.2cm}{.2cm}\par\addvspace{.5cm}}

\def\m{\mu}

\def\1{\textbf{1}}
\def\0{\textbf{0}}

\journal{ABC}

\begin{document}

\begin{frontmatter}



\title{On digraphs determined by their singular values}
\author{Mushtaq A. Bhat$^{a}$}
\author{Peer Abdul Manan$^{b}$}

\address{Department of  Mathematics,  National Institute of Technology, Srinagar-190006, India}
\address {$^{a}$mushtaqab@nitsri.ac.in;~~~$^{b}$\text{mananab214@gmail.com}}

\begin{abstract}
 Let $D$ be an digraph of order $n$ with adjacency matrix $A(D)$ and outdegree matrix $\Delta^+=\Delta^+(D)$. Then the Laplacian and signless Laplacian matrices of $D$ are respectively defined as $L(D)=\Delta^+-A(D)$ and $Q(D)=\Delta^++A(D)$. In this paper, we compute singular values and an exact formula for the trace norm of Laplacian matrices of the directed path $\overrightarrow{P_n}$, the directed cycle $\overrightarrow{C_n}$ and all orientations of a star. We show that for a bipartite digraph $D$, the matrices $L(D)$ and $Q(D)$ have same singular values and use this to compute the singular values and trace norm of signless Laplacian matrices. We study the problem of determination of digraphs by their singular values and prove the directed path $\overrightarrow{P_n}$, the directed cycle $\overrightarrow{C_n}$ and oriented star $\overrightarrow{S}_n(n-1,0)$ are determined by their Laplacian and signless Laplacian singular values but are not determined by their adjacency singular values.
\end{abstract}

\vskip 0.2 true cm

\begin{keyword} Digraph, Singular value, Trace norm, Laplacian Matrix, Signless Laplacian Matrix.

\vskip 0.2 true cm


$MSC$: 05C20, 05C50 

\end{keyword}

\end{frontmatter}

\section{\bf Introduction}
A directed graph (or simply digraph) is a pair $D=(\mathcal{V}, \mathcal{A})$, where $\mathcal{V}$ is a finite set whose elements are called vertices and $\mathcal{A}$ is a set of ordered pairs of elements of $\mathcal{V}$. The elements of $\mathcal{A}$ are known as arcs. An oriented graph is digraph which has no pair of symmetric arcs. We assume our digraphs are simple i.e., free from parallel arcs and loops. A graph $G$ can be identified with a bi-directed graph or symmetric digraph $\overleftrightarrow {G}$ obtained by replacing each edge $e$ of $G$ by a pair of symmetric arcs. In a digraph $D$, an arc from a vertex $u$ to $v$ will be denoted by $uv$. In this case, we say $u$ is the tail and $v$ is the head of arc $uv$. For a vertex $u$ in a digraph $D$, the number of arcs with vertex $u$ as the tail is known as the outdegree of $u$ and we denote it by $d_u^+$ and the number of arcs with $u$ as the head is known as the indegree of $u$ and we denote it by $d_u^{-}$. If $D$ is a digraph of order $n$ with $k (1\le k\le n)$ distinct outdegrees in decreasing order as $d_1^+, d_2^+,\dots, d_k^+ $ having respective multiplicities $m_1, m_2,\dots m_k$, then we write the outdegree sequence of $D$ as $[{d_1^+}^ {m_1}, {d_2^+}^ {m_2},\dots, {d_k^+}^ {m_k}]$. Ganie and Pirzada \cite{gp} defined the first outdegree Zagreb index of a digraph $D$ of order $n$ as $Zg^+(D)=\sum\limits_{i=1}^n(d_i^+)^2$. A digraph $D$ is weakly connected if its underlying graph (obtained by removing the directions on arcs and replacing each cycle of length two by an edge) is connected. A digraph $D$ is strongly connected if for any two vertices $u$ and $v$ of $D$, there are directed paths from $u$ to $v$ and $v$ to $u$. We call a digraph to be connected if it is weakly or strongly connected.\\
\indent Let $D$ be a digraph with vertex set $\mathcal{V}=\{v_1,v_2,\dots,v_n\}$. Then the adjacency matrix $A(D)=(a_{ij})$ of $D$ is a square matrix of order $n$ with $a_{ij}=1$ if there is an arc from vertex  $v_i$ to vertex $v_j$ and zero, otherwise. Let ${\Delta}^+={\Delta}^+(D)=\mbox{diag}~(d^+_1,d^+_2,d^+_3,\dots,d^+_n)$, where $d^+_i=d^+_{v_i}$, be the outdegree matrix of $D$. Then the Laplacian and signless Laplacian matrices of $D$ are respectively defined as $L(D)={\Delta}^{+}-A(D)$ and $Q(D)={\Delta}^{+}+A(D)$.
Given a matrix $N\in M_n(\mathbb{C})$, the matrices $NN^{*}$ and $N^{*}N$ are positive semidefinite and have same nonnegative eigenvalues. The positive square roots of eigenvalues of $NN^{*}$ (or $N^{*}N$) are called singular values of $N$. The trace norm of a complex matrix $N\in M_n(\mathbb{C})$ is defined as 
\begin{equation*}
	\|{N} \|_* =\sum\limits_{i=1}^{n}\sigma_i(N),
\end{equation*}
where $\sigma_1(N)\ge\sigma_2(N)\ge \sigma_3(N)\ge \dots\ge \sigma_n(N)$ are the singular values of $N$ i.e., the positive square roots of eigenvalues of $NN^{*}$.\\
We denote the Laplacian and signless Laplacian singular value respectively by $\sigma^L$ and $\sigma^Q$. If a singular value $\sigma$ is repeated say $k$ times, we indicate this by $\sigma^{(k)}$.

If $N=A(G)$, the adjacency matrix of graph $G$ and $\lambda_i$, where $i=1,2,\dots,n$ are eigenvalues of  the adjacency matrix of a graph $G$, then $\sigma_i(N)=|\lambda_i(G)|$ and so trace norm coincides with the energy of a graph defined by Gutman (1978). Energy of graphs is well studied problem and for more on energy of graphs see \cite{lsg}. Trace norm of a matrix is also known as Nikiforov's energy of a matrix \cite{ar,n1}. So, graph energy extends to digraphs via trace norm as well. By $\mathcal{T}(n)$ and $\mathcal{U}(n)$, we respectively denote the set of oriented trees on $n$ vertices and set of unicyclic digraphs on $n$ vertices. Agudelo and Rada \cite{ar} obtained lower bounds for trace norm of digraphs and characterized the extremal digraphs. Monsalve et al. \cite{mr2}  obtained oriented graphs with maximum and minimum trace norm in the class of oriented bicyclic graphs. Bhat and Manan \cite{bm} studied trace norm for $A_{\alpha}$ matrices of digraphs.\\
Let $\overrightarrow{P_n}$ denote the directed path with $n$ vertices $v_1,v_2,\dots,v_n$ and arc set $\mathcal{A}=\{v_1v_2,v_2v_3,\dots,\\v_{n-1}v_n\}$, let $\overrightarrow{C_n}$ denote a directed cycle with $n$ vertices  $v_1,v_2,\dots,v_n$ and arc set $\mathcal{A}=\{v_1v_2,v_2v_3,\dots\\, v_{n-1}v_n, v_nv_1\}$. For $y\ge 0$ and $0\le x\le n-1$, let $\overrightarrow{S}_n(x,y)$ denote the oriented star graph on $n=x+y+1$ vertices and  with $x$ arcs going out of center vertex and $y$ arcs coming towards center vertex of star. In section 2, we obtain singular values and an exact formula for trace norm of Laplacian matrix of $\overrightarrow{P_n}$, $\overrightarrow{C_n}$ and $\overrightarrow{S}_n(x,y)$. It shown that for a bipartite digraph, Laplcain and signless Laplacian singular values coincide. As an application of this result, we obtain the trace norm of signless Laplacian matrices of $\overrightarrow{P_n}$, $\overrightarrow{C_n}$ and $\overrightarrow{S}_n(x,y)$.
Recall that a  graph $G$ of order $n$ is said to be determined by its spectrum if there exists no non-isomorphic graph of order $n$ with same spectrum as that of $G$. This problem is known as spectral determination problem and is a well studied problem. For example see \cite{abdn,dl,dh,dh1,ts} for few references. In section 3, we show that $\overrightarrow{P_n}$, $\overrightarrow{C_n}$ and $\overrightarrow{S}_n(n-1,0)$ are determined by Laplacian and signless Laplacian singular values but not by adjacency singular values.
We need the following preliminary results.

\begin{lemma} \label{1.1}\cite{hj} (Schur complement) Let \( M \in {M}_{p+q}(\mathbb{R}) \) be a block matrix
	\[
	M = 
	\begin{bmatrix}
		A_{p \times p} & B_{p \times q} \\
		C_{q \times p} & D_{q \times q}
	\end{bmatrix}.
	\]
	
	If \( D \) is invertible, then the Schur complement of \( D \) in \( M \) is the \( p \times p \) matrix defined by
	\[
	M / D = A - BD^{-1}C.
	\]
	
	Also, the determinant satisfies
	\[
	\det(M) = \det(D) \cdot \det(M / D).
	\] 
\end{lemma}
The next result is known as the interlacing property of singular values.
\begin{lemma} \label{1.2} \cite{hj} (Interlacing property) Let $A\in M_n(\mathbb{C})$ and $B$ be a principal submatrix of $A$ of order $r$, then $$\sigma_k(A)\ge \sigma_k(B)\ge \sigma_{k+r}(A),$$
	where $k=1,2,\dots,n$ and we treat $\sigma_j=0$ if $j>n$.
\end{lemma} 
\section{\bf Singular values and trace norm of  Laplacian and signless Laplacian matrices}\label{sec2}
We begin this section by computing the singular values and trace norm of the Laplacian matrix of the directed path $\overrightarrow{P_n}$ and directed cycle $\overrightarrow{C_n}$.
\begin{theorem} \label{2.1} Let $\overrightarrow{C_{n}}$ denote a directed cycle of order $n$ and let $\overrightarrow{P_{n}}$ denote a directed path of order $n$. Then\\
	$(i)$. The singular values of $L(\overrightarrow{C_{n}})$ are $2\sin(\frac{\pi j}{n})$, where $j=0,1,2,\dots,n-1$\\ and $\|L(\overrightarrow{C_n})\|_*=2 \cot\left( \frac{\pi}{2n} \right)$, {for all } $n \ge 2$.\\
	$(ii)$. The singular values of $L(\overrightarrow{P_{n}})$ are $2\sin(\frac{\pi j}{2n})$, where $j=0,1,2,\dots,n-1$\\ and $\|L(\overrightarrow{P_n})\|_*=\cot\left( \frac{\pi}{4n}\right)-1$, {for all } $n \ge 2$.\\
	$(iii)$. 
	$\|L(\overrightarrow{C_n})\|_* = \|L(\overrightarrow{P_n})\|_* + 1 - \tan\left( \frac{\pi}{4n} \right)$.
\end{theorem}
{\bf Proof}. Let $A({\overrightarrow{C_{n}}})$ be the adjacency matrix of $\overrightarrow{C_{n}}$ and $
\Delta^{+}(\overrightarrow{C_n})$ be its out-degree matrix. Note that
$\Delta^{+}(\overrightarrow{C_n}) = I_n$.\\ We have
\begin{align*}
	L(\overrightarrow{C_n}) &= \Delta^{+}(\overrightarrow{C_n}) - A(\overrightarrow{C_n}) \\
	\text {so that}~~LL^{T} &= (\Delta^{+} - A)(\Delta^{+} - A)^{T} \\
	&= (I_n - A)(I_n - A)^{T} \\
	&= I_n - (A + A^{T}) + AA^{T} \\
	&= I_n - A(\overleftrightarrow{C_n}) + I_n \\
	&= 2I_n - A(\overleftrightarrow{C_n})
\end{align*}

The eigenvalues of $2I_n$ are $2^{[n]}$ and eigenvalues of $A({\overleftrightarrow{C_{n}}})$ are $2\cos{\frac{2\pi j}{n}}$ where $j=0,1,2,\cdots ,n-1$.\\
Thus the eigenvalues of $LL^{T}$ are $2-2\cos{\frac{2\pi j}{n}}=4\sin^{2}\frac{\pi j}{n}$, where $j=0, 1, 2, \cdots, n-1$.\\
The singular values of $L({\overrightarrow{C_{n}}})$ are $2\sin{\frac{\pi j}{n}}$ where $j=0,1,2,\cdots ,n-1.$\\
We next compute an exact formula for the trace norm of $L(\overrightarrow{C}_{n})$.
\begin{align*}
	\|L(\overrightarrow{C_n})\|_* 
	&= 2 \sum_{k=0}^{n-1} \sin\left(\frac{\pi k}{n}\right) \\
	&= 2 \cdot \frac{\sin\left(\frac{n\pi}{2n}\right)}{\sin\left(\frac{\pi}{2n}\right)} \cdot \sin\left( \frac{0 + (n-1)\pi/n}{2} \right) \\
	&= 2 \cdot \frac{\sin\left(\frac{\pi}{2}\right)}{\sin\left(\frac{\pi}{2n}\right)} \cdot \sin\left( \frac{(n-1)\pi}{2n} \right) \\
	&= 2 \cdot \csc\left( \frac{\pi}{2n} \right) \cdot \cos\left( \frac{\pi}{2n} \right) \\
	&= 2 \cot\left( \frac{\pi}{2n} \right), \quad \text{for all } n \ge 2.
\end{align*}

$(ii)$. We next compute the singular values and trace norm of  $L(\overrightarrow{P_n})$\\
The outdegree matrix of the directed path \( \overrightarrow{P_n} \) is given by
\[
\Delta^{+}(\overrightarrow{P_n}) = 
\begin{bmatrix}
	I_{n-1} & {\bf 0}_{(n-1) \times 1} \\
	{\bf 0}_{1 \times (n-1)} & 0_{1 \times 1}
\end{bmatrix}.
\]

Let \( A(\overrightarrow{P_n}) \) denote the adjacency matrix of \( \overrightarrow{P_n} \). We note that
\[
A A^{T} = 
\begin{bmatrix}
	I_{n-1} & {\bf 0}_{(n-1) \times 1} \\
	{\bf 0}_{1 \times (n-1)} & 0_{1 \times 1}
\end{bmatrix}
= \Delta^{+}(\overrightarrow{P_n})=(\Delta^{+}(\overrightarrow{P_n}))^2.
\]

It is easy to see  that
\[
\Delta^{+} A^T + A \Delta^{+} =
\begin{bmatrix}
	A(\overleftrightarrow{P}_{n-1}) & {\bf 0}_{(n-1) \times 1} \\
	{\bf 0}_{1 \times (n-1)} & 0_{1 \times 1}
\end{bmatrix}
\]
We have
\begin{align*}
	LL^T &= (\Delta^{+})^2 - \Delta^{+}A^T- A \Delta^{+} + AA^T \\
	&= 2\Delta^{+} - (\Delta^{+}A^T + A\Delta^{+}) \\
	&= 
	\begin{bmatrix}
		2I_{n-1} - A(\overleftrightarrow{P}_{n-1}) & {\bf 0}_{(n-1)\times 1} \\
		{\bf 0}_{1\times(n-1)} & 0_{1\times 1}
	\end{bmatrix}
\end{align*}

Note that the eigenvalues of \( A(\overleftrightarrow{P_n}) \) are
\[
\lambda_j = 2\cos\left( \frac{\pi j}{n+1} \right), \quad j = 1, 2, \dots, n
\]

so that the eigenvalues of \( A(\overleftrightarrow{P}_{n-1}) \) are
\[
\lambda_j = 2\cos\left( \frac{\pi j}{n} \right), \quad j = 1, 2, \dots, n - 1.
\]
Thus, the eigenvalues of \( LL^T \) are

\[
2 - 2\cos\left(\frac{\pi j}{n}\right) = 4\sin^2\left(\frac{\pi j}{2n}\right), \quad j = 0, 1, \ldots, n-1.
\]

Therefore, the singular values of \( L(\overrightarrow{P_n}) \) are

\[
\sigma_j = 2\sin\left(\frac{\pi j}{2n}\right), \quad j = 0, 1, \ldots, n-1
\]

\bigskip

Now, the trace norm of $L(\overrightarrow{P_n})$ is given by
\begin{align*}
	\|L(\overrightarrow{P_n})\|_* 
	&= 2 \sum_{j=0}^{n-1} \sin\left(\frac{\pi j}{2n}\right) \\
	&= 2 \cdot \frac{\sin\left(\frac{n\pi}{4n}\right)}{\sin\left(\frac{\pi}{4n}\right)} \cdot \sin\left(\frac{(n-1)\pi}{4n}\right) \\
	&= 2 \cdot \frac{\sin\left(\frac{\pi}{4}\right)}{\sin\left(\frac{\pi}{4n}\right)} \cdot \sin\left( \frac{\pi}{4} - \frac{\pi}{4n} \right) \\
	&= \frac{2 \cdot \frac{1}{\sqrt{2}}}{\sin\left(\frac{\pi}{4n}\right)} \cdot \left( \cos\left(\frac{\pi}{4n}\right) - \sin\left(\frac{\pi}{4n}\right) \right) \\
	&= \frac{\sqrt{2}}{\sin\left(\frac{\pi}{4n}\right)} \cdot \frac{1}{\sqrt{2}} \left( \cos\left(\frac{\pi}{4n}\right) - \sin\left(\frac{\pi}{4n}\right) \right) \\
	&= \cot\left( \frac{\pi}{4n} \right) - 1.
\end{align*}

(iii).
We next derive a relation between the trace norms of \(L( \overrightarrow{C_n}) \) and \(L(\overrightarrow{P_n})\).\\
We know
\[
\|L(\overrightarrow{C_n})\|_* = 2 \cot\left( \frac{\pi}{2n} \right)
\quad \text{and} \quad
\|L(\overrightarrow{P_n})\|_* = \cot\left( \frac{\pi}{4n} \right) - 1.
\]
Therefore, 
\begin{align*}
	\|L(\overrightarrow{C_n})\|_* - \|L(\overrightarrow{P_n})\|_*
	&= 2 \cot\left( \frac{\pi}{2n} \right) - \left( \cot\left( \frac{\pi}{4n} \right) - 1 \right) \\
	&= 2 \cot\left( \frac{\pi}{2n} \right) - \cot\left( \frac{\pi}{4n} \right) + 1.
\end{align*}

Using the identity
\[
\cot(x) = \frac{\cot(x/2) - \tan(x/2)}{2},~~
\text{we have}~~ 
2 \cot\left( \frac{\pi}{2n} \right) = \cot\left( \frac{\pi}{4n} \right) - \tan\left( \frac{\pi}{4n} \right)
\]

so that,
\[
\|L(\overrightarrow{C_n})\|_* - \|L(\overrightarrow{P_n})\|_* = 1- \tan\left( \frac{\pi}{4n} \right).
\]

Hence,
\[
\|L(\overrightarrow{C_n})\|_* = \|L(\overrightarrow{P_n})\|_* + 1 - \tan\left( \frac{\pi}{4n} \right).
\]
Note that for $n\ge 2$, the sequence $<\|L(\overrightarrow{C_n})\|_* - \|L(\overrightarrow{P_n})\|_*>$ is positive and strictly increasing and converges to 1. We conclude that $\|L(\overrightarrow{C_n})\|_* > \|L(\overrightarrow{P_n})\|_* $.\qed
The singular values of $L(\overrightarrow{P_n})$ and $L(\overrightarrow{C_n})$ give useful identities involving sine function.

Let \( D \) be a digraph with Laplacian matrix $L(D) = \Delta^{+}(D) - A(D)$ and let the singular values of \( L(D) \) be \( \sigma_1^{L}, \sigma_2^{L}, \ldots, \sigma_n^{L} \geq 0 \). Then $\text{Trace}(LL^T) = \sum_{i=1}^{n} (d_i^{+})^2 + \sum_{i=1}^{n} d_i^{+} = Zg^{+}(D) + a(D)$ or $\sum_{i=1}^{n} (\sigma_i^{L})^2 = Zg^{+}(D) + a(D)$.
From the  Laplacian singular values of the cycle $ \overrightarrow{C_n}$, we see
\[
\sum_{j=0}^{n-1} 4 \sin^2\left( \frac{j\pi}{n} \right) = 4 \sum_{j=0}^{n-1} \sin^2\left( \frac{j\pi}{n} \right) = 4 \cdot \frac{n}{2} = 2n.
\]
or~~
\[
\sum_{j=0}^{n-1} \sin^2\left( \frac{j\pi}{n} \right) = \frac{n}{2}.
\]
From the Laplacian singular values of the Path $ \overrightarrow{P_n}$, we see that
\[
\sum_{j=0}^{n-1} 4 \sin^2\left( \frac{j\pi}{2n} \right) = 4 \cdot \sum_{j=0}^{n-1} \sin^2\left( \frac{j\pi}{2n} \right).
\]

This gives~~
\[
\sum_{j=0}^{n-1} 4 \sin^2\left( \frac{j\pi}{2n} \right) = 2(n-1)
\quad \text{or} \quad \sum_{j=0}^{n-1} \sin^2\left( \frac{j\pi}{2n} \right) = \frac{n-1}{2}.
\]

Like eigenvalues, the singular values of Laplacian and signless Laplacian matrix of a digraph are also same and can be seen in the next result.
\begin{proposition}\label{2.2} Let $B$ be a bipartite digraph of order $n$. Then the matrices $L(B)$ and $Q(B)$ have same singular values and hence same trace norm.
\end{proposition}
{\bf Proof.} Let $B_{a,b}$ be a bipartite digraph with partite sets of size $a$ and $b$ such that $a+b=n$. Let $\Delta^+(B)$ be the outdegree matrix of $B$. Then $L(B)=\Delta^+(B)-A(B)$ and $Q(B)=\Delta^+(B)+A(B)$ are respectively Laplacian and signless Laplacian matrices associated to $B$. Define a nonsingular matrix $S$ of order $n$ as 
\[
S = 
\begin{bmatrix}
	I_{a} & {\bf 0}_{a \times b} \\
	{\bf 0}_{b \times a} & -I_{b}
\end{bmatrix}.
\]
Clearly $S=S^T=S^{-1}$. With suitable labelling, the adjacency matrix $A(B)$ of bipartite digraph $B$ can be written as 

\[
\begin{bmatrix}
	{\bf 0}_{a\times a} & {C}_{a \times b} \\
	{H}_{b \times a} & {\bf 0}_{b\times b}
\end{bmatrix}
\]

Partition matrix $\Delta^+(B)$ as 
\[
\Delta^+(B)= 
\begin{bmatrix}
	\Delta^+_a & {\bf 0}_{a \times b} \\
	{\bf 0}_{b \times a} & \Delta^+_b
\end{bmatrix},
\]

We have

\begin{align*}
	S^{-1}Q(B)S&=\begin{bmatrix}
		I_{a} & {\bf 0}_{a \times b} \\
		{\bf 0}_{b \times a} & -I_{b}
	\end{bmatrix}\begin{bmatrix}
		\Delta^+_{a} & {C}_{a \times b} \\
		{H}_{b \times a} & {\Delta^+}_{b}
	\end{bmatrix}\begin{bmatrix}
		I_{a} & {\bf 0}_{a \times b} \\
		{\bf 0}_{b \times a} & -I_{b}
	\end{bmatrix}\\
	&=\begin{bmatrix}
		\Delta^+_{a} & {-C}_{a \times b} \\
		{-H}_{b \times a} & {\Delta^+}_{b}
	\end{bmatrix}=\Delta^+(B)-A(B)=L(B).
\end{align*}

Consequently, $L(B)$ and $Q(B)$ have same singular values and hence same trace norm.\qed

We next obtain the singular values and trace norm of signless Laplacian matrix of directed cycle $\overrightarrow{C_{n}}$ and directed path $\overrightarrow{P_{n}}$.
\begin{theorem} \label{2.3} Let $\overrightarrow{C_{n}}$ denote a directed cycle of order $n$ and let $\overrightarrow{P_{n}}$ denote a directed path of order $n$. Then\\
	$(i)$. The singular values of $Q(\overrightarrow{C_{n}})$ are $2|\cos(\frac{\pi j}{n})|$, where $j=0,1,2,\dots,n-1$\\ and $\|Q(\overrightarrow{C_n})\|_*=2 \cot\left( \frac{\pi}{2n} \right)$, {for all } $n \ge 2$ and $n$ is even and $\|Q(\overrightarrow{C_n})\|_*=2 \csc\left( \frac{\pi}{2n} \right)$, {for all } $n \ge 3$ and $n$ is odd.\\
	$(ii)$. The singular values of $Q(\overrightarrow{P_{n}})$ are $2\sin(\frac{\pi j}{2n})$, where $j=0,1,2,\dots,n-1$\\ and $\|Q(\overrightarrow{P_n})\|_*=\cot\left( \frac{\pi}{4n}\right)-1$, {for all } $n \ge 2$.\\
	$(iii)$. 
	$\|Q(\overrightarrow{C_n})\|_* = \|Q(\overrightarrow{P_n})\|_* + 1 + \tan\left( \frac{\pi}{4n} \right)$.\\
	$(iv)$. 
	$\|Q(\vec{C}_n)\|_* \ge \|L(C_n)\|_*$
	with equality if and only if $n$ is even.
\end{theorem}
{\bf Proof.} We note that from [Lemma 2.1, \cite{bm}], the set of singular values of $Q(\overrightarrow{P_n})$ are 
$A_1=\{2\cos\frac{j\pi}{2n}:j=1,2,3,\ldots,n\}$. As the set of singular values of $L(\overrightarrow{P_n})$ are 
$A_2=\{2\sin\frac{j\pi}{2n}:j=0,2,3,\ldots,n-1\}$. By Proposition $2.2$, we see that $A_1=A_2$ and $\|Q(\overrightarrow{P_n})\|_*=\|L(\overrightarrow{P_n})\|_*=\cot(\frac{\pi}{4n})-1.$

By [Lemma 2.2, \cite{bm}], the set of singular values of $Q(\overrightarrow{C_n})$ is $\{2|\cos(\frac{j\pi}{n})|\}$. In case $n$ is even, then by Proposition $2.2$, $\|Q(\overrightarrow{C_n})\|_*=\|L(\overrightarrow{C_n})\|_*=2\cot(\frac{\pi}{2n}), n\ge 2$.\\
We next consider the case when $n$ is odd.\\

\begin{align*}
	\|Q(\overrightarrow{C_n})\|_* 
	&= 2 \sum_{j=0}^{2k} |\cos(\frac{j\pi}{2k+1})|\\
	&=2\sum_{j=0}^k\cos(\frac{j\pi}{2k+1})-2\sum_{j=k+1}^{2k}\cos(\frac{j\pi}{2k+1}) 
\end{align*}

Using the identity $\sum_{j=0}^{n-1}\cos(a+jd)=\frac{\sin(\frac{nd}{2})}{\sin(\frac{d}{2})}\cdot \cos(a+\frac{(n-1)d}{2})$, and simplifying, we see

\begin{align*}
	2\sum_{j=0}^k\cos(\frac{j\pi}{2k+1})&=\sin(\frac{k\pi}{2k+1})\cot(\frac{\pi}{2(2k+1)})+2{\cos}^2(\frac{k\pi}{2(2k+1)})\\
	&=\cos(\frac{\pi}{2n})\cot(\frac{\pi}{2n})+1+\sin(\frac{\pi}{2n}).
\end{align*}
Also, 
\begin{align*}
	-2\sum_{j=k+1}^{2k}\cos(\frac{j\pi}{2k+1})&=2\sum_{j=1}^{k}\cos(\frac{j\pi}{2k+1})\\
	&=-2+2\sum_{j=0}^{k}\cos(\frac{j\pi}{2k+1})\\
	&=-2+\cos(\frac{\pi}{2n})\cot(\frac{\pi}{2n})+1+\sin(\frac{\pi}{2n}).
\end{align*}
Hence,
\begin{align*}
	\|Q(\overrightarrow{C_n})\|_*&=-2+2+2\cos(\frac{\pi}{2n})\cot(\frac{\pi}{2n})+2\sin(\frac{\pi}{2n})\\
	&=2\csc(\frac{\pi}{2n}).
\end{align*}

We next prove the relation between $\|Q(\overrightarrow{C_n})\|_*$ and $\|Q(\overrightarrow{P_n})\|_*$.\\
For $n$ odd, using the identity $\csc 2x=\frac{1+\cot^2 x}{2\cot x}$, we have
\begin{align*}
	\|Q(\overrightarrow{C_n})\|_*-\|Q(\overrightarrow{P_n})\|_*&=2\csc(\frac{\pi}{2n})+1-\cot(\frac{\pi}{4n})\\
	&=1+\tan (\frac{\pi}{4n}).
\end{align*}

For $n$ even, it is clear that 	$\|Q(\vec{C}_n)\|_* =\|L(C_n)\|_*$. For $n$, odd we see
\begin{align*}
	\|Q(\vec{C}_n)\|_* - \|L(\vec{C}_n)\|_* &= 2 \csc\left(\frac{\pi}{2n}\right) - 2 \cot\left(\frac{\pi}{2n}\right) \\
	&= 2 \left[ \frac{1}{\sin\left(\frac{\pi}{2n}\right)} - \frac{\cos\left(\frac{\pi}{2n}\right)}{\sin\left(\frac{\pi}{2n}\right)} \right] \\
	&= 2 \csc\left(\frac{\pi}{2n}\right) \left[ 1 - \cos\left(\frac{\pi}{2n}\right) \right] \\
	&= 2 \tan\left(\frac{\pi}{4n}\right)>0.
\end{align*}\qed

\begin{theorem}
	If $y\ge 1$ and $0\le x\le n-2$, then
	the singular values of \( L(\overrightarrow{S_n}(x, y))\) (or \( Q(\overrightarrow{S_n}(x, y))\)) are\\ $$0^{[x]}, 1 ^{[y-1]}~~ \text{and}~~ \sqrt{ \frac{(x + x^2 + y + 1) \pm \sqrt{(x + x^2 + y + 1)^2 - 4(x + x^2 + xy)}}{2}}$$
	and 
	\[
	\|L(\overrightarrow{S_n}(x, y)\|_* =\|Q(\overrightarrow{S_n}(x, y)\|_*= (y - 1) + \sqrt{ \frac{(x + x^2 + y + 1) + \sqrt{(x + x^2 + y + 1)^2 - 4(x + x^2 + xy)}}{2}}.
	\]
	If $x=n-1$ and $y=0$, then the singular values of 
	$L( \overrightarrow{S}_n(n - 1, 0) )$ (or $Q( \overrightarrow{S}_n(n - 1, 0) )$) are
	$0^{[n-1]}, \sqrt{(n - 1)^2 + (n - 1)}$
	and $$\|L(\overrightarrow{S}_n(n - 1, 0))\|_*=\|Q(\overrightarrow{S}_n(n - 1, 0))\|_*=\sqrt{(n - 1)^2 + (n - 1)}.$$
\end{theorem}

{\bf Proof.} We prove this result for the Laplacian matrix only. For signless Laplacian matrix, the result follows by Proposition $2.2$. We first compute the Laplacian singular values of  $\overrightarrow{S_n}(x, y)$  with  $x + y = n - 1, y\ge 1$ and $0\le x\le n-2$. We first define some matrices

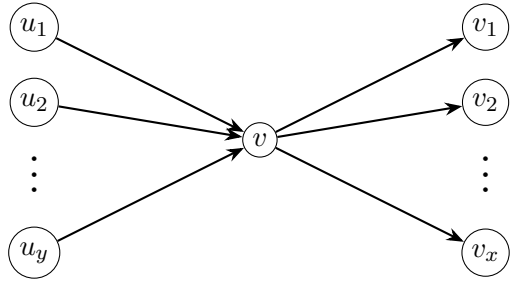
\begin{figure}[h]
	\centering 
	\begin{tikzpicture}[
		vertex/.style={draw, circle, inner sep=2pt, font=\small},
		edge/.style={-Stealth, thick}
		]
		
		\node[vertex] (v) at (0,0) {$v$};
		
		\node[vertex] (u1) at (-3, 1.5)  {$u_1$};
		\node[vertex] (u2) at (-3, 0.5)  {$u_2$};
		\node[scale=1.5] (udot) at (-3, -0.3) {$\vdots$};
		\node[vertex] (uy) at (-3, -1.5) {$u_y$};
		
		\node[vertex] (v1) at (3, 1.5)  {$v_1$};
		\node[vertex] (v2) at (3, 0.5)  {$v_2$};
		\node[scale=1.5] (vdot) at (3, -0.3) {$\vdots$};
		\node[vertex] (vx) at (3, -1.5) {$v_x$};
		
		\draw[edge] (u1) -- (v);
		\draw[edge] (u2) -- (v);
		\draw[edge] (uy) -- (v);
		
		\draw[edge] (v) -- (v1);
		\draw[edge] (v) -- (v2);
		\draw[edge] (v) -- (vx);
		
	\end{tikzpicture}
	\caption{An oriented Star  $\vec{S}_n(x,y)$}
\end{figure}

Let \( J_y \) be the all-ones matrix of order \( y \times y \), $E$ be a square matrix of order $x+1$ with
\[
E_{11} = 1, \quad E_{ij} = 0 \text{ for } (i,j) \ne (1,1)
\]

and $F$ be a square matrix of order $x+1$ with
\[
F_{1j} = 1 \text{ for } j = 2, 3, \ldots, x+1; \quad F_{ij} = 0 \text{ otherwise,}
\]

and $G$ be a matrix of order $y\times (x+1)$ with
\[
G_{i1} = 1 \text{ for } i = 1, 2, \ldots, y; \quad G_{ij} = 0 \text{ for } (i,j) \ne (1,1), (2,1), \ldots, (y,1). 
\]
The out-degree matrix \( \Delta^{+}(\overrightarrow{S_n}(x, y) \) is
\[
\Delta^{+}(\overrightarrow{S_n}(x, y)) = \mathrm{diag}({x}, \underbrace{0,\dots,0}_{x\text{ times}}, \underbrace{1,\dots,1}_{y\text{ times}})
=
\Delta^+(\overrightarrow{S_n}(x, y)) =
\begin{bmatrix}
	xE & \mathbf{0} \\
	\mathbf{0} & I_{y}
\end{bmatrix}.
\]
Also,
\[
A(\overrightarrow{S_n}(x, y)) =
\begin{bmatrix}
	F & \mathbf{0} \\
	G & \mathbf{0}
\end{bmatrix},
\]
so that
\[
A\Delta^+ =
\begin{bmatrix}
	\mathbf{0} & \mathbf{0} \\
	xG & \mathbf{0}
\end{bmatrix}
\quad \text{and} \quad
\Delta^+A^T =
\begin{bmatrix}
	\mathbf{0} & {xG^T} \\
	\mathbf{0} & \mathbf{0}
\end{bmatrix},
\]
and therefore
\[
\Delta^+A^T + A\Delta^+ =
\begin{bmatrix}
	\mathbf{0} & xG^T \\
	xG & \mathbf{0}
\end{bmatrix},
\]
Also,
\[
AA^T =
\begin{bmatrix}
	xE & \mathbf{0} \\
	\mathbf{0} & J_y
\end{bmatrix},
\]
and
\[
(\Delta^+)^2 =
\begin{bmatrix}
	x^2E & \mathbf{0} \\
	\mathbf{0} & J_y
\end{bmatrix}.
\]
We have
\[
LL^T = (\Delta^+)^2 - (\Delta^+A^T + A\Delta^+) + AA^T =
\begin{bmatrix}
	(x^2 + x)E & -xG^T \\
	-xG & (J + I)_y
\end{bmatrix}
\]
so that
\[
\phi_{LL^T}(\lambda) = \det(\lambda I_n - LL^T)
= \det
\begin{bmatrix}
	\lambda I_{x+1} - (x^2 + x)E & xG^T \\
	xG & \lambda I_y - (J + I)
\end{bmatrix}
\]

Let
\( A = \lambda I_{x+1} - (x^2 + x)E \),~~
\( B = xG^T \),~~
\( C = xG \) and ~~
\( D = \lambda I_y - (J + I) \).

By Lemma $1.1$, we see
\[
\phi_{LL^T}(\lambda) = \det(A) \cdot \det(D - CA^{-1}B).
\]
As
\( A^{-1} = \operatorname{diag}\left( \frac{1}{\lambda - (x^2 + x)}, \frac{1}{\lambda}, \dots, \frac{1}{\lambda} \right) \)
\( CA^{-1}B = \frac{x^2}{\lambda - (x^2 + x)} J_y \).

So that
\[
D - CA^{-1}B = (\lambda - 1)I_y - \frac{x^2}{\lambda - (x^2 + x)} J_y.
\]
The eigenvalues of \( D - CA^{-1}B \) are
\[
\lambda - 1 - \frac{(\lambda - x)y}{\lambda - x - x^2}, \quad (\lambda - 1)^{[y-1]}.\]

Hence,
\[
\det(D - CA^{-1}B) = (\lambda - 1)^{y - 1} \left[ (\lambda - 1) - \frac{(\lambda - x)y}{\lambda - x - x^2} \right].
\]

Now,
\[
\phi_{LL^T}(\lambda) = \lambda^x (\lambda - 1)^{y - 1} (\lambda - x - x^2) \left[ (\lambda - 1) - \frac{(\lambda - x)y}{\lambda - x - x^2} \right]
\]
\[
=\left[\lambda^x  (\lambda - 1)^{y-1}][(\lambda - 1)(\lambda - x - x^2) - (\lambda - x)y \right].
\]
The eigenvalues of \( LL^T \) are
\( 0 \) with multiplicity \( x \),
\( 1 \) with multiplicity \( y - 1 \)
and two nontrivial eigenvalues given by
\[
\frac{(x + x^2 + y + 1) \pm \sqrt{(x + x^2 + y + 1)^2 - 4(x + x^2 + xy)}}{2}.
\]
Thus, the singular values of $(L(\overrightarrow{S_n}(x, y))$ are
\[
\underbrace{0, \dots, 0}_{x \text{ times}}, \quad \underbrace{1, \dots, 1}_{y - 1 \text{ times}}, \quad \sqrt{ \frac{(x + x^2 + y + 1) \pm \sqrt{(x + x^2 + y + 1)^2 - 4(x + x^2 + xy)}}{2} }.
\]
and the trace norm of $L(\overrightarrow{S_n}(x, y) )$ is
\[
\text{Trace norm} = (y - 1) + \sqrt{ \frac{(x + x^2 + y + 1) + \sqrt{(x + x^2 + y + 1)^2 - 4(x + x^2 + xy)}}{2} }.
\]
We next consider the case
\( y = 0 \), so \( x = n - 1 \). 
\[
\phi_{LL^T}(\lambda) = \lambda^{n - 1} (\lambda + a_1)
\quad \text{where} \quad a_1 = -(\sigma_1^L)^2.
\]

Also, since the sum of squares of Laplacian singular values equals the trace of \( LL^T \), we have
\[
(\sigma_1^L)^2 = Zg^+(D) + a = (n - 1)^2 + (n - 1)
~\text{gives}~ \sigma_1^L = \sqrt{(n - 1)^2 + (n - 1)}.
\]

So, the singular values of 
$L( \overrightarrow{S}_n(n - 1, 0) )$ are
$0^{[n-1]}, \sqrt{(n - 1)^2 + (n - 1)}$
and $$\|L(\overrightarrow{S}_n(n - 1, 0))\|_*=\sqrt{(n - 1)^2 + (n - 1)}.$$\qed

\section{\bf Digraphs determined by Laplacian and signless Laplacian singular values}\label{sec3}
We  say a digraph $D$ of order $n$ is determined by its singular values if there exists no digraph of order $n$, non-isomorphic to $D$ and has same singular values as that of $D$. The adjacency singular values of $\overrightarrow{C}_n$ are $[1^{(n)}]$ and those of $\overrightarrow{P}_n$ are $[1^{(n-1)},0]$. For example see [Lemma 2.1, \cite{bm}] with $\alpha=0$. We note that $\overrightarrow{C}_{n-1}\bigcup K_1$  and $\overrightarrow{P}_n$ have same adjacency ($A$) singular values. It is clear that both $\overrightarrow{P}_n$ and $\overrightarrow{C}_n$ are not determined by adjacency singular values. The adjacency singular values of ${\overrightarrow S}_n(x,y)$ are $[\sqrt{x}, \sqrt{y}, 0^{(n-2)}]$ \cite{mr2}. It is clear that ${\overrightarrow S}_n(x,y)$ is not determined by its adjacency singular values. For example for $x\neq y$, ${\overrightarrow S}_n(x,y)$ and ${\overrightarrow S}_n(y,x)$ are non-isomorphic but share same adjacency singular values.\\
We next show that both the digraphs $\overrightarrow{P}_n$ and $\overrightarrow{C}_n$ are determined by their Laplacian ($L$) and signless Laplacian ($Q$) singular values. We also show that an oriented star ${\overrightarrow S}_n(x,y)$ is not in general determined by its Laplacian and signless Laplacian singular values. As the oriented stars ${\overrightarrow S}_n(0,n-1)$ and ${\overrightarrow S}_n(1,n-2)$ are clearly non-isomorphic. By Theorem $2.4$, both share same Laplacian (respectively signless Laplacain) singular values $[\sqrt{n}, 1^{(n-2)}, 0]$. However as shown in next result the oriented star ${\overrightarrow S}_n(n-1,0)$ is determined by its Laplacian and signless Laplacian singular values.
\begin{theorem} \label{3.1} 
	$(i)$. The digraphs $\overrightarrow{P}_n$ and $\overrightarrow{C}_n$ are determined by their Laplacian singular values.
	$(ii)$. The digraphs $\overrightarrow{P}_n$ and $\overrightarrow{C}_n$ are determined by their signless Laplacian singular values.\\
	$(iii)$. The digraph $\overrightarrow{S}_n(n-1,0)$ is determined by its Laplacian and signless Laplacian singular values.
\end{theorem}
{\bf Proof.} $(i)$. The Laplacian singular values of the directed cycle $\overrightarrow{C}_n$ are given by
\[ \sigma_{j}^{L} = 2 \sin\left(\frac{\pi j}{n}\right), \quad j = 0, 1, 2, \dots, n-1. \]
Therefore,
\begin{align*}
	\sum_{j=1}^{n} (\sigma_{j}^{L})^2 &= \text{Tr}\left(L(\overrightarrow{C}_n) L^T(\overrightarrow{C}_n)\right) \\
	&= \sum_{i=1}^{n} (d_i^{+})^2 + \sum_{i=1}^{n} d_i^{+} \\
	&= {Zg}^+(\overrightarrow{C}_n) + a(\overrightarrow{C}_n) \\
	&= n + n = 2n,
\end{align*}
where $a(\overrightarrow{C}_n)$ represents the number of arcs. 
We first compare $\overrightarrow{C}_n$ with oriented trees on $n$ vertices. The outdegree sequence of $\overrightarrow{C}_n$ is $D_1=[1^n]$. We note that there is no oriented tree $T \in \mathcal{T}(n)$ with outdegree sequence $[1^n]$. For each oriented tree, the number of arcs is $a=n-1$. Consequently, the maximum number of positive outdegrees an oriented tree can possess is $n-1$, yielding an outdegree sequence of $[1^{n-1}, 0]$.  
For any such oriented tree $T$ with outdegree sequence $D_2=[1^{n-1}, 0]$, we have $Zg^{+}(T) = n-1$, $a(T) = n-1$ and 
$\sum_{j=1}^{n} (\sigma_j^{L}(T))^2 < \sum_{j=1}^{n} (\sigma_{j}^{L}(\overrightarrow{C}_n))^2 $.
Therefore, an oriented tree $T$ with outdegree sequence $D_2=[1^{n-1},0]$ and the directed cycle $\overrightarrow{C}_n$ cannot have the same Laplacian singular values.

For oriented trees with $n-2$ positive outdegrees, the possible outdegree sequence is $D_{3} = [2^{1}, 1^{n-3}, 0^2]$. For each oriented tree $T$ with outdegree sequence $D_3$, we see
\begin{align*}
	Zg^{+}(T) &= 2^2 + (n-3)(1^2) = 4 + n - 3 = n+1 
	~~\text{and} ~~	a(T) = n-1.
\end{align*}
It follows that
\[ \sum_{j=1}^{n} (\sigma_j^{L}(T))^2 = Zg^{+}(T) + a(T) = (n+1) + (n-1) = 2n=\sum_{j=1}^{n} (\sigma_{j}^{L}(\overrightarrow{C}_n))^2. \]
In this case, further investigation is required, and we utilize the interlacing property of singular values.\\
For $n=4$, the outdegree sequence $[2, 1, 0^2]$ corresponds to the three distinct oriented trees $T_1, T_2,$ and $T_3$ shown in Figure 2. Using matlab, we see that $\sigma_{1}^{L}(T_1)=2.61$, $\sigma_{1}^{L}(T_2)=2.49$ and $\sigma_{1}^{L}(T_3)=2.49$.

\begin{figure}[ht]
	
	\begin{minipage}{0.3\textwidth}
		\centering
		\begin{tikzpicture}[node distance={15mm}, main/.style = {draw, circle, inner sep=2pt}] 
			\node[main] (v1) {$v_1$}; 
			\node[main] (v2) [above left of=v1] {$v_2$}; 
			\node[main] (v3) [above right of=v1] {$v_3$}; 
			\node[main] (v4) [below of=v1] {$v_4$};
			\draw [->, >={Stealth}] (v4) -- (v1);
			\draw [->, >={Stealth}] (v1) -- (v2);
			\draw [->, >={Stealth}] (v1) -- (v3);
		\end{tikzpicture}
		\centerline{$T_1$}
	\end{minipage}
	\hfill
	\begin{minipage}{0.3\textwidth}
		
		\begin{tikzpicture}[node distance={15mm}, main/.style = {draw, circle, inner sep=2pt}] 
			\node[main] (v1) {$v_1$}; 
			\node[main] (v2) [above left of=v1] {$v_2$}; 
			\node[main] (v3) [above right of=v1] {$v_3$}; 
			\node[main] (v4) [below of=v3] {$v_4$};
			\draw [->, >={Stealth}] (v1) -- (v2);
			\draw [->, >={Stealth}] (v1) -- (v3);
			\draw [->, >={Stealth}] (v3) -- (v4);
		\end{tikzpicture}
		\centerline{$T_2$}
	\end{minipage}
	\hfill
	\begin{minipage}{0.3\textwidth}
		\centering
		\begin{tikzpicture}[node distance={15mm}, main/.style = {draw, circle, inner sep=2pt}] 
			\node[main] (v1) {$v_1$}; 
			\node[main] (v2) [above left of=v1] {$v_2$}; 
			\node[main] (v3) [above right of=v1] {$v_3$}; 
			\node[main] (v4) [below of=v3] {$v_4$};
			\draw [->, >={Stealth}] (v1) -- (v2);
			\draw [->, >={Stealth}] (v1) -- (v3);
			\draw [->, >={Stealth}] (v4) -- (v3);
		\end{tikzpicture}
		\centerline{$T_3$}
	\end{minipage}
	
	\vspace{0.5cm}
	\caption{Oriented trees on 4 vertices with outdegree sequence $[2, 1, 0, 0]$.}
	\label{fig:oriented_trees}
\end{figure}
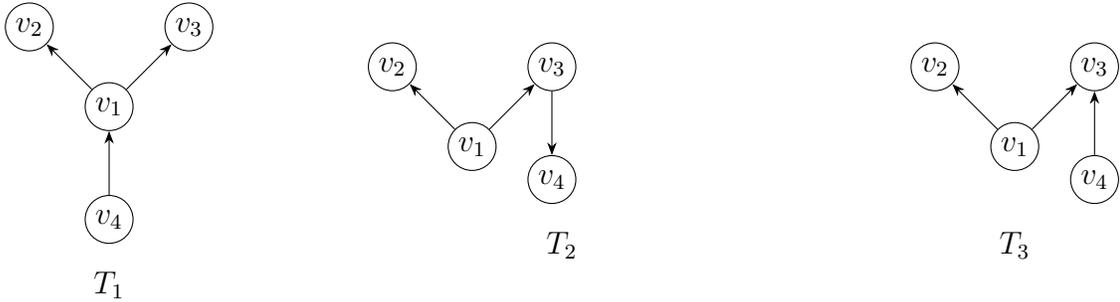

The $L$-spectral norm of the oriented cycle $\overrightarrow{C}_n$ is given by
\begin{equation*}
	\sigma_{1}^L(\overrightarrow{C}_n) = 
	\begin{cases} 
		2 & \text{if } n \text{ is even,} \\
		2 \cos\left(\frac{\pi}{2n}\right) & \text{if } n \text{ is odd.}
	\end{cases}
\end{equation*}

Clearly, $\sigma_{1}^L(T_i) > \sigma_{1}^L(\overrightarrow{C}_4)$ and the result follows for $n=4$.\\
For $n \ge 5$, let $T \in \mathcal{T}(n)$ has the outdegree sequence $D_4 = [2, 1^{n-3}, 0^2]$. The possible principal submatrices of order $4$ of $L(T)$  are $L(T_i)$ for $i \in \{1, 2, 3\}$, where $T_i~i=1,2,3$ are shown in Figure $2$ with $\sigma_{1}^L(T_1)=2.61$, $\sigma_{1}^L(T_2)=2.49$, $\sigma_{1}^L(T_3)=2.49$,\\ 
and\\
$M_1 = L(T_1) + \text{diag}(1, 1, 0, 0),$ with $\sigma_{1}(M_1)=3.44$\\
$M_2 = L(T_1) + \text{diag}(1, 0, 0, 0),$ with  $\sigma_{1}(M_2)=3.45$\\
$M_3 = L(T_1) + \text{diag}(0, 1, 0, 0),$ with  $\sigma_{1}(M_3)=2.64$\\
$M_4 = L(T_2) + \text{diag}(0, 1, 0, 1),$ with  $\sigma_{1}(M_4)=2.53$\\
$M_5 = L(T_2) + \text{diag}(0, 1, 0, 0),$ with  $\sigma_{1}(M_5)=2.53$\\
$M_6 = L(T_2) + \text{diag}(0, 0, 0, 1),$ with $\sigma_{1}(M_6)=2.49$\\
$M_7 = L(T_3) + \text{diag}(0, 1, 1, 0),$ with $\sigma_{1}(M_7)=2.58$\\
$M_8 = L(T_3) + \text{diag}(0, 1, 0, 0),$ with $\sigma_{1}(M_8)=2.53$\\
$M_9 = L(T_3) + \text{diag}(0, 0, 1, 0).$ with $\sigma_{1}(M_9)=2.55$.

We note that $\sigma_{1}^L(T_i)>\sigma_{1}^L(\overrightarrow{C}_n)$ for $i=1,2,3$ and $\sigma_{1}(M_i) > \sigma_{1}^L(\overrightarrow{C}_n)$ for $i=1,2,\dots, 9$. Consequently, $\sigma_{1}^L(T)>\sigma_{1}^L(\overrightarrow{C}_n)$ for all oriented trees with outdegree sequence $D_3$.\\
For all oriented trees $T\in \mathcal{T}(n)$ with at most $n-3$ positive outdegrees, it is easy to see that $Zg^{+}(T)>n+1$ and consequently for any such oriented tree $$\sum_{j=1}^{n} (\sigma_j^{L}(T))^2>(n+1)+(n-1)=2n=\sum_{j=1}^{n} (\sigma_j^{L}(\overrightarrow{C}_n))^2.$$
Hence $\overrightarrow{C}_n$ and any oriented tree $T$ on $n$ vertices cannot have same singular values.\\ 
\noindent We show that no $U \in \mathcal{U}(n)\setminus\{\overrightarrow{C}_n\}$ and $\overrightarrow{C}_n$ have same singular values. Recall that if $x_1,x_2, \cdots x_n$ are non-negative integers such that $\sum_{i=1}^{n}x_{i}=n$ and $ 0 \le x_i \le n-1$, then $\sum_{i=1}^{n}x_{i}^{2}$ is minimized at $x_1=x_2=\dots=x_n=1$ and the minimum value is n. We restrict our comparison to $U \in \mathcal{U}(n) \setminus \{\overrightarrow{C}_n\}$ with outdegree sequence $D_4=[1^n]$ as for other $U \in \mathcal{U}(n) \setminus \{\overrightarrow{C}_n\}$, $Zg^{+}(U)>n$.

For $n=4$, the only such unicyclic digraph with outdegree sequence $[1^4]$ is $U_1$ given in Figure 3. The singular values of $U_1$ are $[2, 1.7321, 1, 0]$ whereas singular values of $L(\overrightarrow{C}_4)$ are $0,2 \sqrt{2},\sqrt{2}$.\\
For $n\ge 5$, such a unicyclic digraph $U$ must contain one of $U_2$, $U_3$, $U_4$, $T_4$ as an induced subdigraph on $5$ vertices and accordingly $L(U_2), L(U_3), L(U_4), L(T_4)$ will be a principal submatrix of $L(U)$. We note that $\sigma_1^L(U_2)=2.2361$, $\sigma_1^L(U_3)=2.042$, $\sigma_1^L(U_4)=2.0743$ and $\sigma_1^L(T_4)=2.0421$. So, by interlacing property $\sigma_1^L(U)>\sigma_1^L(\overrightarrow{C}_n)$. Hence, the result follows in this case.

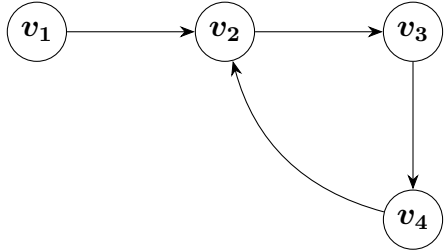
\begin{figure}[ht]
	\centering
	\begin{tikzpicture}[
		node distance={25mm}, 
		main/.style = {draw, circle, font=\boldmath, inner sep=3pt}
		] 
		\node[main] (v1) {$v_1$}; 
		\node[main] (v2) [right of=v1] {$v_2$}; 
		\node[main] (v3) [right of=v2] {$v_3$}; 
		\node[main] (v4) [below of=v3] {$v_4$};
		
		\begin{scope}[>={Stealth[scale=1.2]}, every edge/.style={draw, thick}]
			\draw [->] (v1) -- (v2);
			\draw [->] (v2) -- (v3);
			\draw [->] (v3) -- (v4);
			\draw [->] (v4) to [bend left=30] (v2);
		\end{scope}
		
	\end{tikzpicture}
	\caption{$U_1$, unicyclic digraph  with outdegree sequence $[1^4]$}
\end{figure}

\begin{figure}[ht]
	\centering
	\begin{minipage}{0.48\textwidth}
		\centering
		\begin{tikzpicture}[scale=0.7, node distance={18mm}, main/.style = {draw, circle, font=\small\boldmath, inner sep=2pt}] 
			\node[main] (v3) {$v_3$}; 
			\node[main] (v4) [above right of=v3, node distance=22mm] {$v_4$}; 
			\node[main] (v5) [below right of=v4, node distance=22mm] {$v_5$};
			\node[main] (v1) [above left of=v3] {$v_1$};
			\node[main] (v2) [below left of=v3] {$v_2$};
			\begin{scope}[>={Stealth[scale=1]}, every edge/.style={draw, thick}]
				\draw [->] (v1) -- (v3); \draw [->] (v2) -- (v3);
				\draw [->] (v3) -- (v4); \draw [->] (v4) -- (v5); \draw [->] (v5) -- (v3);
			\end{scope}
		\end{tikzpicture}
		\centerline{$U_2$}
	\end{minipage}
	\hfill
	\begin{minipage}{0.48\textwidth}
		\centering
		\begin{tikzpicture}[scale=0.7, node distance={18mm}, main/.style = {draw, circle, font=\small\boldmath, inner sep=2pt}] 
			\node[main] (v1) {$v_1$}; 
			\node[main] (v2) [right of=v1] {$v_2$}; 
			\node[main] (v3) [right of=v2] {$v_3$}; 
			\node[main] (v4) [above right of=v3, node distance=22mm] {$v_4$}; 
			\node[main] (v5) [below right of=v4, node distance=22mm] {$v_5$};
			\begin{scope}[>={Stealth[scale=1]}, every edge/.style={draw, thick}]
				\draw [->] (v1) -- (v2); \draw [->] (v2) -- (v3);
				\draw [->] (v3) -- (v4); \draw [->] (v4) -- (v5); \draw [->] (v5) -- (v3);
			\end{scope}
		\end{tikzpicture}
		\centerline{$U_3$}
	\end{minipage}
	
	\vspace{1cm} 
	
	\begin{minipage}{0.48\textwidth}
		\centering
		\begin{tikzpicture}[scale=0.7, node distance={20mm}, main/.style = {draw, circle, font=\small\boldmath, inner sep=2pt}] 
			\node[main] (v1) {$v_1$}; 
			\node[main] (v2) [right of=v1] {$v_2$}; 
			\node[main] (v3) [right of=v2] {$v_3$}; 
			\node[main] (v4) [below of=v3] {$v_4$};
			\node[main] (v5) [left of=v4] {$v_5$};
			\begin{scope}[>={Stealth[scale=1]}, every edge/.style={draw, thick}]
				\draw [->] (v1) -- (v2); \draw [->] (v2) -- (v3);
				\draw [->] (v3) -- (v4); \draw [->] (v5) -- (v4);
				\draw [->] (v4) to [bend left=30] (v2);
			\end{scope}
		\end{tikzpicture}
		\centerline{$U_4$}
	\end{minipage}
	\hfill
	\begin{minipage}{0.48\textwidth}
		\centering
		\begin{tikzpicture}[scale=0.7, node distance={18mm}, main/.style = {draw, circle, font=\small\boldmath, inner sep=2pt}] 
			\node[main] (v1) {$v_1$}; 
			\node[main] (v2) [right of=v1] {$v_2$}; 
			\node[main] (v3) [right of=v2] {$v_3$}; 
			\node[main] (v4) [right of=v3] {$v_4$};
			\node[main] (v5) [above of=v3] {$v_5$};
			\begin{scope}[>={Stealth[scale=1]}, every edge/.style={draw, thick}]
				\draw [->] (v1) -- (v2); \draw [->] (v2) -- (v3);
				\draw [->] (v3) -- (v4); \draw [->] (v5) -- (v3);
			\end{scope}
		\end{tikzpicture}
		\centerline{$T_4$}
	\end{minipage}
	
	\vspace{0.5cm}
	\caption{Possible subdigraphs $U_2, U_3, U_4$ and $T_4$ for $n \ge 5$.}
	\label{fig:subdigraphs}
\end{figure}

\noindent If $D$ is a digraph with $n$ vertices and $\ge n+1$ arcs, then $Zg^+({D}) + a({D})>2n$, and hence no such ${D}$ can have the same Laplacian singular values as those of $\overrightarrow{C}_n$. \\

\noindent If ${D}$ is an acyclic digraph with $n$ vertices and $n$ arcs, then the outdegree sequence of ${D}$ differs from $[1^n]$ and $Zg^+({D})>n$.\\

\noindent If ${D}$ is not connected and has $k \ge 2$ components, since $L(\overrightarrow{C}_n)$ has $0$ as a singular value with multiplicity 1, whereas ${D}$ has $0$ as a singular value with multiplicity at least $k \ge 2$ and so $D$ and $\overrightarrow{C}_n$ cannot have same singular values.\\
\indent We next show that $\overrightarrow{P}_n$ is determined by $L$ singular values. We have
\begin{align*}
	\sum_{j=1}^n (\sigma_{j}^L(\overrightarrow{P}_n))^2 &= Zg^+(\overrightarrow{P}_n) + a(\overrightarrow{P}_n) \\
	&= (n-1) + (n-1) = 2n-2.
\end{align*}

\noindent We note that it is sufficient to compare $\overrightarrow{P}_n$ only with those oriented trees whose outdegree sequence is $[1^{n-1}, 0]$. The Laplacian singular values of $\overrightarrow{P}_n$ are
\[ 2\sin\left(\frac{\pi j}{2n}\right); \quad j=0, 1, 2, \dots, n-1, \]
and so $\sigma_1^{L}(\overrightarrow{P}_n)=2\cos{\frac{\pi}{2n}}<2$.\\ 
\noindent For $n=4$, one can verify the result directly. For $n \ge 5$, any $T \in \mathcal{T}(n) \setminus \{\overrightarrow{P}_n\}$ with outdegree sequence $[1^{(n-1)}, 0]$ must contain one of the sub-oriented trees shown in Figure 5. For any such oriented tree $T$, the Laplacian matrix $L(T)$ has one among $ L(T_5),~ L(T_6),~ L(T_7),~ M_{10}=L(T_7)+\text{diag}(1,0,0,0,0)$ as principal submatrix. It is easy to see that $\sigma_1^{L}(T_5)=\sqrt{5},~ \sigma_1^{L}(T_6)=2.0421, ~\sigma_1^{L}(T_7)= 2.0421, ~\sigma_1(M_{10})=2.0491$. By interlacing property, for any oriented tree $T\in\mathcal{T}(n) \setminus \{\overrightarrow{P}_n\}$ with outdegree sequence $[1^{n-1},0]$, $\sigma_1^{L}(T)>2$ and so it has Laplacian singular values different from those of $\overrightarrow{P}_n$.

\begin{figure}[ht]
	\centering
	\begin{minipage}{0.32\textwidth}
		\centering
		\begin{tikzpicture}[scale=0.8, main/.style = {draw, circle, inner sep=2pt, font=\small}]
			\node[main] (v1) {$v_1$};
			\node[main] (v2) [right of=v1] {$v_2$};
			\node[main] (v3) [right of=v2] {$v_3$};
			\node[main] (v4) [above of=v2] {$v_4$};
			\node[main] (v5) [below of=v2] {$v_5$};
			\begin{scope}[>={Stealth[scale=1]}, every edge/.style={draw, thick}]
				\draw [->] (v1) -- (v2);
				\draw [->] (v2) -- (v3);
				\draw [->] (v4) -- (v2);
				\draw [->] (v5) -- (v2);
			\end{scope}
		\end{tikzpicture}
		\centerline{$T_5$}
	\end{minipage}
	\hfill
	\begin{minipage}{0.32\textwidth}
		\centering
		\begin{tikzpicture}[scale=0.8, main/.style = {draw, circle, inner sep=2pt, font=\small}]
			\node[main] (v1) {$v_1$};
			\node[main] (v2) [right of=v1] {$v_2$};
			\node[main] (v3) [right of=v2] {$v_3$};
			\node[main] (v4) [right of=v3] {$v_4$};
			\node[main] (v5) [above of=v3] {$v_5$};
			\begin{scope}[>={Stealth[scale=1]}, every edge/.style={draw, thick}]
				\draw [->] (v1) -- (v2);
				\draw [->] (v2) -- (v3);
				\draw [->] (v3) -- (v4);
				\draw [->] (v5) -- (v3);
			\end{scope}
		\end{tikzpicture}
		\centerline{$T_6$}
	\end{minipage}
	\hfill
	\begin{minipage}{0.32\textwidth}
		\centering
		\begin{tikzpicture}[scale=0.8, main/.style = {draw, circle, inner sep=2pt, font=\small}]
			\node[main] (v1) {$v_1$};
			\node[main] (v2) [left of=v1] {$v_2$};
			\node[main] (v3) [left of=v2] {$v_3$};
			\node[main] (v4) [above left of=v3] {$v_4$};
			\node[main] (v5) [below left of=v3] {$v_5$};
			\begin{scope}[>={Stealth[scale=1]}, every edge/.style={draw, thick}]
				\draw [->] (v2) -- (v1);
				\draw [->] (v3) -- (v2);
				\draw [->] (v4) -- (v3);
				\draw [->] (v5) -- (v3);
			\end{scope}
		\end{tikzpicture}
		\centerline{$T_7$}
	\end{minipage}
	\caption{Oriented trees with outdegree sequence $[1^4, 0]$.}
\end{figure}
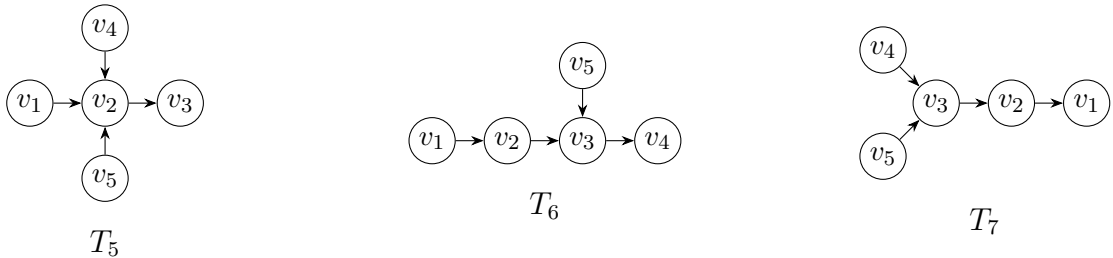
(ii). We will compare $\overrightarrow{C}_n$ with the class of unicyclic digraphs and proof for other cases is similar to part (i). From [Lemma $2.9$, \cite{bm}], we see that for any $U\in\mathcal{U}(n)$, $\sigma_1^Q(U)\ge 2$ with equality if and only if $U=\overrightarrow{C}_n$, hence no $U\in\mathcal{U}(n)\setminus\overrightarrow{C}_n$ and $\overrightarrow{C}_n$ have same $Q$ singular values.\\
The proof that $\overrightarrow{P}_n$ is determined by $Q$ singular values is similar.\\
(iii). By [Theorem 2.5, \cite{bm}], we see that $\overrightarrow{S}_n(n-1,0)$ is the unique digraph on $n$ vertices with rank of signless Laplacian matrix equal to $1$. A similar proof shows that this is also the unique digraph of order $n$ with rank of Laplacian matrix equal to $1$. Hence, $\overrightarrow{S}_n(n-1,0)$ is determined by its Laplacian and signless Laplacian singular values.\qed

\noindent{\bf Acknowledgements.} This research is supported by NBHM project No. 02011/42/2025/NBHM(RP)\\/R\&DII/16567 and by SERB-DST grant with File No. MTR/2023/000201. The research of Peer Abdul Manan is supported by CSIR, New Delhi, India with CSIR-HRDG Ref. No: Jan-Feb/06/21(i)EU-V. 

\end{document}